\documentclass[11pt]{amsart}

\makeindex

\usepackage[all, cmtip]{xy}
\usepackage{amscd,amssymb, amsmath, wasysym}
\usepackage{amsfonts}
\usepackage[pdftex]{graphicx}
\usepackage{multirow}
\usepackage[usenames,dvipsnames,svgnames,table]{xcolor}
\usepackage{enumerate}

\usepackage[margin=1.15in]{geometry}
\usepackage{mathrsfs}    % [00pt]changes size of font
\usepackage{amsthm}     % proclaims theoremstyle/proof environments
\usepackage{eucal}      % Euler Cal/Script Fonts
\usepackage{latexsym}   % latex symbols  (like \pounds)
\usepackage{verbatim}   % for comment environment
\usepackage[all]{xy}     % commutative diagram

\usepackage[bookmarks=true,bookmarksnumbered=true,
 pdftitle={},
 pdfsubject={},
 pdfauthor={},
 pdfkeywords={TeX; dvipdfmx; hyperref; color;},
 colorlinks=true, linkcolor=BrickRed, citecolor=RoyalBlue]
 {hyperref}

\theoremstyle{definition}

\theoremstyle{remark}

\theoremstyle{corollary}

\theoremstyle{theorem}

\theoremstyle{corollary}

\newtheorem{theorem}{Theorem}[section]
\newtheorem{lemma}[theorem]{Lemma}
\newtheorem{proposition}[theorem]{Proposition}

\theoremstyle{corollary}

\theoremstyle{definition}

\theoremstyle{remark}
\newtheorem{remark}[theorem]{Remark}

\numberwithin{equation}{section}

\newcommand{\Z}{\mathbb{Z}}

\def\P{\mathbb{P}}

\def\reg{\operatorname{reg}}

\def\codim{\operatorname{codim}}

\newcommand{\nZ}{\mathbb{Z}}                     % integer number
\newcommand{\nP}{\mathbb{P}}                     % projective space

\newcommand{\sF}{\mathscr{F}}
\newcommand{\sG}{\mathscr{G}}
\newcommand{\sO}{\mathscr{O}}                    % structure sheaf

\newcommand{\sI}{\mathscr{I}}                    % ideal sheaf

\newcommand{\sV}{\mathscr{V}}
\newcommand{\sW}{\mathscr{W}}

\DeclareMathOperator{\Char}{char}                % char
\DeclareMathOperator{\Coker}{Coker}              % Coker
\DeclareMathOperator{\id}{id}                    % id, injective dimension
\DeclareMathOperator{\Sym}{Sym}                  % Sym
\newcommand*{\longsurjrightarrow}{\ensuremath{\joinrel\relbar\joinrel\twoheadrightarrow}}

\newcommand{\suchthat}{\;\ifnum\currentgrouptype=16 \middle\fi|\;}

\input xy
\xyoption{all}

\title[A Castelnuovo-Mumford regularity bound for scrolls]{A Castelnuovo-Mumford regularity bound for scrolls}

\begin{document}

\author{Wenbo Niu}
\address{Department of Mathematical Sciences, University of Arkansas, Fayetteville, AR 72701, USA}
\email{wenboniu@uark.edu}

\author{Jinhyung Park}
\address{School of Mathematics, Korea Institute for Advanced Study, Seoul 02455, Republic of Korea}
\email{parkjh13@kias.re.kr}

\thanks{J.P was partially supported by Basic Science Research Program through the National Research Foundation of Korea (NRF) funded by the Ministry of  Science, ICT and Future Planning (NRF-2016R1C1B2011446).}

\date{\today}
%\date{June 22, 2009 and, in revised form,}

\dedicatory{Dedicated to Professor Lawrence Ein on the occasion of his sixtieth birthday.}


\begin{abstract}
Let $X \subseteq \P^r$ be a scroll of codimension $e$ and degree $d$ over a smooth projective curve of genus $g$.
The purpose of this paper is to prove a linear Castelnuovo-Mumford regularity bound that $\reg(X) \leq d-e+1+g(e-1)$. This bound works  over an algebraically closed field of arbitrary characteristic.
\end{abstract}

\maketitle

%\tableofcontents \setcounter{page}{1}

\section{Introduction}

\noindent Throughout the paper, we work over an algebraically closed field $k$ of arbitrary characteristic.
Let $X \subseteq \P^r$ be a projective variety defined by an ideal sheaf $\sI_{X}$. We say that $X$ is \emph{$m$-regular} if $H^i(\P^r, \sI_{X}(m-i))=0$ for all $i>0$. The minimal such number $m$ is called the \emph{Castelnuovo-Mumford regularity} of $X$ and is denoted by $\reg(X)$. It has attracted considerable attentions in the past thirty years to bound the regularity in terms of geometric or algebraic invariants. One optimal bound has been conjectured by Eisenbud-Goto \cite{EG} that $$\reg(X) \leq \deg X-\codim X+1.$$
Very recently, counterexamples involving singular varieties for the regularity conjecture have been found by McCullough-Peeva \cite{MP}.
However, it is still interesting to study whether this conjecture or a weaker variant holds for important cases, and the conjecture actually has been proven for integral curves by \cite{GLP}, for smooth complex surfaces by \cite{P} and \cite{L}, and certain singular surfaces by \cite{Niu}. Slightly weaker results for lower dimensional smooth varieties in characteristic zero were also obtained in \cite{K1} and \cite{K2}.
As one of the promising cases of the conjecture, nonsingular scrolls of arbitrary dimension were studied in \cite[Theorem 3]{B}, while the proof there contains a miscalculation (see Remark \ref{bertin} for details).
On the other hand, Noma showed that the double point divisor associated to a generic inner projection of a smooth projective variety $X \subseteq \P^r$ is semiample except when $X$ is a scroll, a Roth variety, or the second Veronese surface \cite{N2}, and he proved a weaker bound for the regularity of Roth varieties \cite{N3}. Thus the scroll case is of special interest to us.

Motivated by the work of \cite{GLP} and \cite{B}, we establish a regularity bound for scrolls in this paper.
Precisely, let $C$ be a smooth projective curve of genus $g\geq 0$ and let $E$ be a very ample vector bundle on $C$ of rank $n$ and degree $d$.  The variety $X=\nP(E) \subseteq \P^r=\nP(V)$ embedded by a base-point-free subspace $V \subseteq H^0(\P(E), \sO_{\P(E)}(1))$ is called a \emph{scroll over $C$}.
Note that $X$ is a non-degenerate smooth projective variety of dimension $n$ and degree $d$. The main result is  the following theorem.

\begin{theorem}\label{main1}
Let $X \subseteq \P^r$ be a scroll of codimension $e$ and degree $d$ over a smooth projective curve $C$ of genus $g\geq 0$.  Then one has
 $$\reg(X) \leq d-e+1+g(e-1).$$
\end{theorem}

When $g=0$, Theorem \ref{main1} gives a sharp bound. This was first proved by Bertin since the proof in \cite{B} works in this case and is independent on the characteristic of the base field. Another proof using a different method in characteristic zero was also given in \cite[Theorem 5.2]{KEP}. If $g=1$, we get an interesting bound $\reg(X) \leq d$. If we further impose extra conditions, the above bound can be established for certain singular scrolls (see Remark \ref{rmk:01} for details).

One of the essential points in the proof of Theorem \ref{main1} is to pick a particular line bundle on the curve with certain cohomological properties  and then to resolve its pullback on the scroll as a $\sO_{\nP^r}$-module. This idea goes back to \cite{GLP} and has been used in \cite{B} and \cite{N1}. We point out that the choice of such line bundle is not unique. In \cite{GLP} and \cite{N1}, for instance, a different line bundle with even stronger properties has been considered (see Remark \ref{rmk:02} for this direction in our case).

It is interesting to mention that a sharp regularity bound for the structure sheaf of a smooth projective variety in characteristic zero has been established in \cite[Theorem A]{KJP}, where the characteristic zero and the smoothness assumption are crucial. In some sense, the regularity bound would be relatively easy to establish for structure sheaves.  We point out some evidence in this direction in Corollary \ref{p:02}.

%Bertin claimed that the regularity conjecture holds for scrolls in \cite{B}. However, her proof contains a serious gap (see Remark \ref{bertin}). We closely follow her arguments and fix the gap.

The paper is organized as follows: in Section \ref{prelimsec}, we collect basic facts and properties; Section \ref{proofsec} is devoted to the proof of Theorem \ref{main1}.\\

\noindent{\em Acknowledgment:}
The second author would like to thank the University of Arkansas for their generous hospitality, where this work was started.
The authors wish to thank the anonymous referee for helpful suggestions.

\section{Preliminaries}\label{prelimsec}

\noindent In this section, we briefly review relevant basic facts for the convenience of the reader. By a \emph{variety}, we mean a separated reduced and irreducible scheme of finite type over the field $k$.

A coherent sheaf $\sF$ on a projective variety $X$ with a very ample line bundle $L$ is said to be \emph{$m$-regular} with respect to $L$ if
$$H^i(X, \sF\otimes L^{m-i})=0 \quad \text{for all }i >0.$$
When the line bundle $L$ is clear from the context, we  simply say $\sF$ is $m$-regular.
The least such number, if exits, is denoted by $\reg (\sF)$.  If $X$ is a subvariety of a projective space $\nP^r$, the line bundle $L$ is always assumed to be $\sO_{X}(1)$.
By Mumford's regularity theorem (\cite[Theorem 1.8.5]{positivityI}), if $\sF$ is $m$-regular, then $\sF$ is $(m+1)$-regular. For more details, see \cite[Section 1.8]{positivityI}. We shall use the following property of regularity.

\begin{proposition}[{\cite[Lemmas 2.5]{L}}]\label{regten}
Consider an exact sequence of coherent sheaves on $\P^r$
$$
\cdots \longrightarrow F_i \longrightarrow \cdots \longrightarrow F_1 \longrightarrow F_0 \longrightarrow F \longrightarrow 0.
$$
If $F_i$ is $(p+i)$-regular for some integer $p$ and for each $i \geq 0$, then $F$ is $p$-regular.
%$(2)$ Let $F$ and $G$ be vector bundles on $\P^r$. If $F$ is $p$-regular and $G$ is $q$-regular, then $F \otimes G$ is $(p+q)$-regular. In particular, if $\text{char}(k)=0$, then $\wedge^mE$ and $S^m(E)$ are $mp$-regular.
\end{proposition}

The line bundle $L$ induces an embedding $X \subseteq \P^r$ so that $X$ is defined by an ideal sheaf  $\sI_X$ on $\P^r$. Then $X$ is said to be \emph{$m$-regular} if  $\sI_X$ is $m$-regular. From the definition, that $X$ is $m$-regular is equivalent to the following two conditions:
\begin{enumerate}[\indent$(1)$]
\item $X$ is $(m-1)$-normal, i.e., the natural restriction map
$$
H^0(\P^{r}, \sO_{\P^{r}}(m-1)) \longrightarrow H^0 (X, \sO_X(m-1))
$$
is surjective;
\item The structure sheaf $\sO_X$ is $(m-1)$-regular.
\end{enumerate}
The number $\reg(\sI_X)$ is called the \emph{Castelnuovo-Mumford regularity} of $X$ and is denoted by $\reg(X)$.

Under suitable conditions, the regularity of $X$ can be read off from the regularity of its hyperplane section. We state this observation here.
\begin{lemma}\label{reglem}
Let $X \subseteq \P^r$ be a non-degenerate projective variety of dimension $n \geq 2$. Suppose that a general hyperplane section $Y =X\cap \P^{r-1}$ is $m$-regular for some integer $m>0$. Then $\sO_X$ is $(m-1)$-regular. If moreover $X$ is $(m-1)$-normal, then it is $m$-regular.
\end{lemma}

\begin{proof}
Consider the short exact sequence
$$
0 \longrightarrow  \sO_X(-1) \longrightarrow \sO_X \longrightarrow \sO_Y \longrightarrow 0.
$$
Since $Y \subseteq \P^{r-1}$ is $k$-normal for all $k \geq m-1$, we obtain the surjective map
$$
H^0(X, \sO_X(k)) \longsurjrightarrow H^0(Y, \sO_Y(k)).
$$
Furthermore, since $\sO_Y$ is $k$-regular for all $k \geq m-1$, we have $$H^{i-1}(Y, \sO_Y(k-i+1))=H^i(Y, \sO_Y(k-i+1))=0 \quad  \text{for all }i \geq 2.$$
It now follows that
$$
H^i(X, \sO_X(k-i)) = H^i(X, \sO_X(k-i+1)) \quad \text{for all }i \geq 1.
$$
Note that $H^i(X, \sO_X(l))=0$ for $i>0$ and a sufficiently large integer $l$.
Thus $\sO_X$ is $(m-1)$-regular. If further assume that $X$ is $(m-1)$-normal, then it is clear that $X$ is $m$-regular.
\end{proof}

In some cases, it is relatively easy to obtain an optimal regularity bound for the structure sheaf, as showed in the following proposition.
\begin{proposition}\label{p:02}
	Let $X \subseteq \P^r$ be a non-degenerate projective variety of codimension $e$ and degree $d$ over an algebraically closed field $k$. Assume one the following conditions holds:
	\begin{itemize}
		\item [(i)] $\dim X=2$;
		\item [(ii)] $\dim X=3$, $X$ has isolated singularities and $\Char k=0$.
	\end{itemize}
	Then $\reg(\sO_X) \leq d-e.$
\end{proposition}
\begin{proof}
	Take a general hyperplane section $Y = X \cap \P^{r-1}$.
	Since we are working over an infinite field, we can apply Bertini's theorem to $Y$. Thus $Y$ is an integral curve under the condition (i) and a nonsingular surface under the condition (ii). The embedding $Y \subset \P^{r-1}$ is non-degenerate. By \cite{GLP} and \cite{L}, we know that $\reg(Y) \leq d-e+1$. %Consider the short exact sequence
%	$$
%	0 \longrightarrow  \sO_X(-1) \longrightarrow \sO_X \longrightarrow \sO_Y \longrightarrow 0.
%	$$
%	Since $Y \subset \P^{r-1}$ is $k$-normal for all $k \geq d-e$, we obtain the surjective map
%	$$
%	H^0(X, \sO_X(k)) \longsurjrightarrow H^0(Y, \sO_Y(k)).
%	$$
%	Furthermore, since $\sO_Y$ is $k$-regular for all $k \geq d-e$, it follows that
%	$$
%	H^1(X, \sO_X(k-1))=H^1(X, \sO_X(k)) \text{ and } H^2(X, \sO_X(k-2))=H^2(X, \sO_X(k-1)).
%	$$
%	Note that $H^i(X, \sO_X(l))=0$ for $i>0$ and a sufficiently large integer $l >0$.
	Now by the lemma above, $\sO_X$ is $(d-e)$-regular.
\end{proof}

Next, we briefly recall the definition of Koszul groups associated to a line bundle of a projective variety. For the details, we refer the reader to Green's paper  \cite{Green:KoszulI}.
Let $V \subseteq H^0(X, L)$ be a base-point-free subspace. There is a canonical surjective evaluation homomorphism $e_{V} \colon V \otimes \sO_X \rightarrow L$. Define $M_{V}$ as the kernel of $e_V$ to form a short exact sequence
$$
0 \longrightarrow M_V \longrightarrow V \otimes \sO_X \stackrel{e_V}{\longrightarrow} L \longrightarrow 0.
$$
Fix a line bundle $B$ on $X$, and define
$$
R=R(X, B, L)= \bigoplus_{m \geq 0} H^0(X,B+ mL).
$$
Then $R$ is naturally a graded $S=\text{Sym}^\bullet (V)$-module so that it has a minimal free resolution
$$
\cdots \longrightarrow \bigoplus_{q} K_{p,q}(X, B, V) \otimes_k S(-p-q) \longrightarrow \cdots  \longrightarrow \bigoplus_{q} K_{0,q}(X, B, V) \otimes_k S(-q)  \longrightarrow R \longrightarrow 0,
$$
in which the vector space $K_{p,q}(X, B, V)$ is the Koszul group associated to $B$ with respect to $L$. We shall use the following proposition to compute these groups.

\begin{proposition}[{\cite[Proposition 3.2]{EL2}}]\label{koscoh}
Assume that $H^i(X, B+mL)=0$ for $i>0$ and $m>0$. Then for $q \geq 2$, we have
$$
K_{p,q}(X, B, V)=H^1(X, \wedge^{p+1}M_V \otimes (B+(q-1)L)).
$$
If moreover $H^1(X, B)=0$, then we allow $q=1$ in the above formula.
\end{proposition}

We conclude this section by listing a couple of technical results that we shall use in our proof.

\begin{lemma}[{cf. \cite[Lemma 1.7]{GLP}, \cite[Lemma 4.1]{B}}]\label{speciallb}
Let $C$ be a smooth irreducible projective curve of genus $g$, and $p \colon C \to \P^{e+1}$ be a  morphism of degree $d=\deg p^*\sO_{\nP^{e+1}}(1)$. Assume that $p(C)$ is non-degenerate, and set $M=p^* \Omega^1_{\P^{e+1}}(1)$. Then for a general line bundle $A$ of degree $d-e+g$ on $C$, one has $h^0(C, A)=d-e+1$ and $h^1(C, A)=h^1(C, M \otimes A)=h^1(C, \wedge^2 M \otimes A)=0$.
\end{lemma}

\begin{proof}
As in \cite[Proof of Lemma 1.7]{GLP}, we first consider a filtration
$$
M=F^1 \supseteq F^2 \supseteq \cdots \supseteq F^{e+1} \supseteq F^{e+2}=0
$$
of $M$ by vector bundles such that each of the quotients $L_i = F^i/F^{i+1}$ is a line bundle of strictly negative degree. Then $h^1(C, L_i \otimes A)=h^1(C, L_i \otimes L_j \otimes A)=0$ implies $h^1(C, M \otimes A)=h^1(C, \wedge^2 M \otimes A)=0$.
Since $\deg M = -d$ and $\text{rank } M = e+1$, it follows that $\deg L_i \geq e-d$ and $\deg L_i \otimes L_j \geq -d+e-1$. Note that a generic line bundle of degree $\geq g-1$ is non-special. Thus the assertion now follows.
\end{proof}

\begin{lemma}[{\cite[Theorem B.2.2]{positivityI}}]\label{lm:10} Let  $u \colon E \rightarrow F$ be a generically surjective homomorphism of vector bundles of ranks $e$ and $f$ on a smooth variety. Associated to $u$, the \emph{Eagon-Northcott complex}
\begin{equation*}
{\bf EN_u:}
\begin{split}
0 \to \wedge^eE \otimes \Sym^{e-f-1}(F^*) \otimes \det F^* \to \cdots \hspace{7cm}\\
\cdots \to \wedge^{f+2}E \otimes F^* \otimes \det F^* \to \wedge^{f+1}E \otimes \det F^* \to E \xrightarrow{u} F \to 0.
\end{split}
\end{equation*}	
is exact away from the support of $\text{coker}(u)$.
\end{lemma}

\begin{lemma}\label{dialem}
	Let $\sF$ and $\sG$ be coherent sheaves on $\nP^r$. Consider the following diagram
	\begin{equation}\label{eq:04}
	\begin{split}
	\xymatrix{
		\sO_{\P^r}^a \ar[r]^u \ar@{^{(}.>}[d]^{\exists\ \phi'} & \sF \ar[r]\ar[d]^\phi & 0 \\
		\sO_{\P^r}^b \ar[r]^v & \sG \ar[r] & 0
	}
	\end{split}
	\end{equation}
	where $u$ and $v$ are surjective morphisms, $\phi$ is a morphism, and $a\leq b$. Assume that $u$, $v$ and $\phi$ induce injective morphisms on global sections and satisfy the following condition $$(\phi\circ u)(H^0(\nP^r,\sO_{\P^r}^a))\subseteq v(H^0(\nP^r, \sO_{\P^r}^b))\subseteq H^0(\nP^r,\sG).$$ Then there exists an injective morphism $\phi':\sO_{\P^r}^a\hookrightarrow \sO_{\P^r}^b$ lifting $\phi$ with $\Coker \phi'\simeq \sO_{\nP^r}^{b-a}$ such that the diagram (\ref{eq:04}) is commutative.
\end{lemma}
\begin{proof} By assumption, we may assume that $\sO_{\P^r}^a\simeq V\otimes \sO_{\nP^r}$ and $\sO_{\P^r}^b\simeq W\otimes \sO_{\nP^r}$ for vector subspaces $V$ and $W$ of $H^0(\nP^r,\sG)$ such that $V\subseteq W$. Hence $\phi'$ is induced by the inclusion $V\subseteq W$. The commutativity of the diagram (\ref{eq:04}) is checked locally by using the fact that at a closed point $x$, the elements of $V$ and $W$ induce  sets of generators for $\sF_x$ and $\sG_x$, respectively.
\end{proof}

%The following diagram chasing lemma will be useful.
%
%\begin{lemma}[{\cite[Proposition B.1.2]{positivityI}}]\label{diagramchasing}
%Consider a complex
%$$
%\sF_{n-1} \to \cdots \to \sF_1 \to \sF_0 \xrightarrow{f} \sG \to 0
%$$
%of coherent sheaves on a projective variety $X$ of dimension $n$ with $f$ surjective. Assume that\\
%\indent$(1)$ the complex is exact away from a set of dimension $\leq 1$; \\
%\indent$(2)$ for a given integer $1 \leq m \leq n$, we have $h^i(X, \sF_0) = \cdots = h^i(X, \sF_{n-m})=0$ for $i>0$.\\
%Then $h^i(X, \sG)=0$ for $i \geq m$.
%\end{lemma}

\section{Proof of the main result}\label{proofsec}

\noindent This section is devoted to the proof of Theorems \ref{main1}. Recall that $C$ is a nonsingular irreducible projective curve of genus $g\geq 0$ and $E$ is a very ample vector bundle on $C$ of rank $n$ and degree $d$. The scroll $X=\nP(E)$ is embedded into a projective space $\nP^r=\nP(V)$ by a map $p$  determined by a base-point-free subspace $V \subseteq H^0(X, \sO_{X}(1))$ of dimension $r+1$. Denote by $t:X\rightarrow C$ the canonical projection. Write $e=\codim X$ and set $S=\Sym^\bullet (V)$ as the symmetric algebra of $V$. Our setting can be summarized in the following diagram
\[
\begin{split}
\xymatrix{
	X=\P(E) \ar[d]_t \ar[r]^-p& \P^r\\
	C &
}
\end{split}.
\]
%Just for the convenience, we use notation $X$ to view the scroll as a porjecive subvariety in the projective space and $\nP(E)$ to emphasize it is a projective bundle over the curve. More often, we ignore the embedding $p$  if no confusion arises.\\

\begin{proof}[Proof of Theorem \ref{main1}] First of all, the vector space $V$ is naturally a base-point-free subspace of $H^0(C,E)$ under the identification $H^0(C,E)=H^0(X,\sO_X(1))$. Let $M$ be the kernel of the surjective evaluation map $V\otimes \sO_X\rightarrow \sO_X(1)$. Then there is a short exact sequence
\begin{equation}\label{eq:41}
0\longrightarrow M\longrightarrow V\otimes\sO_X\longrightarrow \sO_X(1)\longrightarrow 0.
\end{equation}
Since $R^1t_*M=0$, pushing down this sequence to $C$ by $t$ yields a short exact sequence
\begin{equation}
0\longrightarrow t_*M\longrightarrow V\otimes\sO_C\longrightarrow E\longrightarrow 0.
\end{equation}

Take a general linear subspace $W\subseteq V$ of dimension $n-1$ and let $\overline{V}=V/W$ be the quotient space. The linear space $\nP(\overline{V})$ is naturally a linear subspace of $\nP^r$. As $W$ is general, the following conditions hold.
\begin{enumerate}
	\item [\stepcounter{equation}(\theequation)]There is a short exact sequence $$0\longrightarrow W\otimes\sO_C\longrightarrow E\longrightarrow \det E\longrightarrow 0.$$
	\item [\stepcounter{equation}(\theequation)] The intersection $C_0=X\cap \nP(\overline{V})$ is a curve such that the restriction map $$t|_{C_0}:C_0\longrightarrow C$$ is an isomorphism.
\end{enumerate}
Snake lemma gives rise to the following diagram
$$\begin{CD}
@. @. 0  @. 0 @. \\
@. @. @VVV @VVV\\
@. @. W \otimes \sO_C @=  W \otimes \sO_C \\
@. @. @VVV @VVV\\
0 @>>> t_*M @>>> V \otimes \sO_C @>>> E @>>> 0\\
@. @| @VVV @VVV\\
0 @>>> t_*M @>>> \overline{V}\otimes \sO_C @>>> \det E @>>> 0\\
@. @. @VVV @VVV\\
@. @. 0  @. 0 @. \\
\end{CD}.$$
We identify $C$ as $C_0$ so that  $\sO_{\nP^r}(1)|_{C}=\det E$ and the embedding $C\subseteq \nP(\overline{V})$ is determined by the subspace $\overline{V}\subseteq H^0(C,\det E)$. Therefore  $$t_*M=\Omega^1_{\nP(\overline{V})}|_{C}.$$

We apply Lemma \ref{speciallb} to the curve $C\subseteq \nP(\overline{V})$. So there exists a general line bundle $A$ on $C$ of degree $d-r+n+g$ satisfying the properties:
\begin{eqnarray}\label{vanA}
& & h^1(C, A)=h^1(C, t_*M \otimes A) = h^1(C, \wedge^2 t_*M \otimes A)=0,\hspace{2cm}\\ \label{eq:55}
& & h^0(C, A)=d-r+n+1,  \\ \label{eq:56}
& & h^0(C, t_*M \otimes A) = (r-n)(d-r+n)+1.
%h^1(C, A)=h^1(C, M \otimes A) = h^1(C, \wedge^2 M \otimes A)=0, \\
%h^0(C, A)=d-r+n+1, \quad\mbox{and } h^0(C, M \otimes A) = (r-n)(d-r+n)+1.
\end{eqnarray}
Since $d-r+n+g\geq g$, we can further assume $A$ is effective, i.e.,
\begin{equation}\label{eq:11}
A= \sO_C(D) \text{ \quad where } D=\sum_{i=1}^{d-r+n+g} x_i, \text{ with distinct general points $x_i\in C$.}
\end{equation}
%It is easy to compute the following vanishings of cohomology:
%\begin{eqnarray}\label{eq:31}
%& & H^i(X, t^*A \otimes \sO_{X}(m))=0,\text{ for }0 < i < n \text{ and }m \in \Z;\quad\quad \quad\quad\\\label{eq:32}
%& & H^n(X, t^*A \otimes \sO_{X}(m))=0,\text{ for } m>1-n;\\
%& & H^0(X, t^*A \otimes \sO_{X}(-j))=0, \text{ for } j>0.
%\end{eqnarray}

Consider the graded $S$-module
$$
R=\bigoplus_{m \geq 0} H^0(X, t^*A \otimes \sO_{X}(m))
$$
Note that $H^0(X, t^*A \otimes \sO_{X}(m))=0$, for $m<0$. We claim that the line bundle $t^*A$ is $1$-regular (with respect to $\sO_X(1)$). To see this, it is enough to verify $H^i(X,t^*A\otimes \sO_X(1-i))=0$ for $1\leq i\leq n$. Since $R^jt_*\sO_X(l)=0$ for $l\geq 1-n$ and $j\geq 1$,  by the projection formula we have $H^i(X,t^*A\otimes \sO_X(1-i))=H^i(C,A\otimes S^{1-i}E)$. The latter one is automatically zero if $i>1$ and for $i=1$, $H^1(C,A)=0$ by the choice of $A$. Hence $t^*A$ is indeed $1$-regular. So $R$ is $1$-regular too as a graded $S$-module. Thus a minimal free resolution of $R$ over $S$ should have the following shape
\begin{equation}\label{eq:05}
\begin{split}
0\longrightarrow S^{\alpha_{r-n}}(-r+n)\oplus S^{\beta_{r-n}}(-r+n-1)\longrightarrow\cdots \longrightarrow S^{\alpha_2}(-2)\oplus S^{\beta_2}(-3) \longrightarrow \hspace{1cm}\\
\longrightarrow S^{\alpha_1}(-1)\oplus S^{\beta_1}(-2)  \longrightarrow S^{\alpha_0}\oplus S^{\beta_0}(-1) \longrightarrow R \longrightarrow 0.
\end{split}
\end{equation}
%where for $i=0,1,\cdots$, $r-n$, each $F_i$ has the form $$F_i=S'^{\alpha_i}(-i)\oplus S'^{\beta_i}(-i-1),$$ for some $\alpha_i,\ \beta_i\geq 0$.
%$$
%0\longrightarrow F_{r-n}\longrightarrow\cdots \longrightarrow F_2 \longrightarrow H^0(C, M \otimes A) \otimes S'(-1) \longrightarrow H^0(C, A) \otimes S' \longrightarrow R\otimes S' \longrightarrow 0,
%$$
Using Proposition \ref{koscoh} and properties (\ref{vanA})-(\ref{eq:56}), we can determine that $\beta_0=\beta_1=0$ and
$$\alpha_0=h^0(C,A)=d-r+n+1,$$
$$\alpha_1=h^0(C,t_*M\otimes A)=(r-n)(d-r+n)+1.$$
%$$
%\cdots \longrightarrow S^{\alpha_2}(-2)\oplus S^{\beta_2}(-3)\longrightarrow H^0(\P(E), M' \otimes t^*A) %\otimes S(-1) \longrightarrow H^0(\P(E), t^*A) \otimes S \longrightarrow R \longrightarrow 0.
%$$
After sheafification, we then obtain a free resolution of $t^* A$
\begin{equation}\label{eq:03}
\cdots \longrightarrow\sO_{\nP^r}^{\alpha_2}(-2)\oplus\sO_{\nP^r}^{\beta_2}(-3) \longrightarrow \sO_{\P^r}^{(r-n)(d-r+n)+1}(-1) \longrightarrow \sO_{\nP^r}^{d-r+n+1}  \longrightarrow t^* A \longrightarrow 0.
\end{equation}

\bigskip

Recall that $A=\sO_C(D)$ for an effective divisor $D$. Denote by  $s_D\in H^0(C, A)$ the corresponding global section.  Then there is a short exact sequence
$$
0 \longrightarrow \sO_C \stackrel{\cdot s_D}{\longrightarrow} A \longrightarrow A_D \longrightarrow 0.
$$
Let $X_i := t^{-1}(x_i) \simeq \P^{n-1}$. Pulling back this sequence by $t$ gives  rise to a short exact sequence
$$
0 \longrightarrow \sO_{X} \stackrel{\cdot s_D}{\longrightarrow} t^*A \longrightarrow \sF=\bigoplus_{i=1}^{d-r+n+g} \sO_{X_i} \longrightarrow 0,
$$
where $\sF$ is a direct sum of $\sO_{X_i}$ as indicated.
Resolving each $\sO_{X_i}$ by the Koszul resolution, we obtain a free resolution of $\sF$ as follows:
\begin{equation}\label{eq:06}
\begin{split}
0\longrightarrow \sO_{\P^r}^{n_{r-n+1}}(-r+n-1)\longrightarrow\cdots  \hspace{8cm} \\
\longrightarrow \sO_{\P^r}^{n_2}(-2) \stackrel{w''}{\longrightarrow} \sO_{\P^r}^{(r-n+1)(d-r+n+g)}(-1) \stackrel{w'}{\longrightarrow} \sO_{\P^r}^{d-r+n+g} \stackrel{w}{\longrightarrow} \sF \longrightarrow 0.
\end{split}
\end{equation}
On the other hand by Snake lemma (see also \cite[Proof of Theorem 2.1]{GLP}), we obtain the following commutative diagram
\begin{equation}\label{eq:01}
\begin{split}
\xymatrix{
 & & & 0 \ar[d] & 0 \ar[d] & \\
 & 0 \ar[d] &  & \sO_{\P^r}  \ar[r] \ar[d] & \sO_X \ar[r] \ar[d]^-{\cdot s_D} & 0 \\
0 \ar[r] & K \ar[r] \ar[d] & \sO_{\P^r}^{(r-n)(d-r+n)+1}(-1) \ar[r]\ar@{=}[d] & \sO_{\P^r}^{d-r+n+1} \ar[r] \ar[d] & t^* A \ar[r] \ar[d] & 0\\
0 \ar[r] & N \ar[r] \ar[d] & \sO_{\P^r}^{(r-n)(d-r+n)+1}(-1) \ar[r]^-{v'} & \sO_{\P^r}^{d-r+n} \ar[r]^-{v} \ar[d] & \sF \ar[r] \ar[d] & 0\\
& \sI_X \ar[d] & & 0 & 0 & \\
& 0 &&&&
}
\end{split}
\end{equation}
in which all horizontal and vertical sequences are exact.
A truncation of the free resolution (\ref{eq:03}) induces a free resolution of the sheaf $K$, from which we
see that $H^1(\P^r, K(m))=H^2(\P^r, K(m))=0$ for all $m \in \Z$. Therefore from the left vertical sequence in the diagram (\ref{eq:01}), we get
\begin{equation}\label{eq:10}
H^1(\P^r, N(m))=H^1(\P^r, \sI_X(m)) \quad \mbox{ for all } m\in \nZ.
\end{equation}

\bigskip

In the diagram (\ref{eq:01}), the sheaf $\sF$ is also resolved by the exact sequence
$$0 \longrightarrow N \longrightarrow \sO_{\P^r}^{(r-n)(d-r+n)+1}(-1) \stackrel{v'}{\longrightarrow} \sO_{\P^r}^{d-r+n} \stackrel{v}{\longrightarrow} \sF \longrightarrow 0.$$
We point out that the kernel $N$ of $v'$ may not be locally free. We shall compare it with the free resolution (\ref{eq:06}) of $\sF$. Our goal is to lift the identity morphism $\id_{\sF}$ of $\sF$ consecutively to construct two injective morphisms $\phi_0$ and $\phi_1$ to achieve the following commutative diagram
\begin{equation}\label{eq:07}
\begin{split}
\xymatrix{
	0 \ar[r] & N \ar@{^{(}->}[d]^{\phi_2} \ar[r] & \sO_{\P^r}^{(r-n)(d-r+n)+1} (-1) \ar@{^{(}->}[d]^-{\phi_1} \ar[r]^-{v'} & \sO_{\P^r}^{d-r+n}  \ar@{^{(}->}[d]^-{\phi_0} \ar[r]^-v & \sF \ar[r] \ar@{=}[d]^-{\id_{\sF}} & 0\\
	0 \ar[r] & P  \ar[r] &  \sO_{\P^r}^{(r-n+1)(d-r+n+g)} (-1) \ar[r]^-{w'} & \sO_{\P^r}^{d-r+n+g}  \ar[r]^-w & \sF \ar[r] & 0
}
\end{split},
\end{equation}
where $P=\ker(w')$ and $\phi_2=\phi_1|_N$. For this, let us construct $\phi_0$ first. Set $\sV=\ker v$ and $\sW=\ker w$.
It is easy to check that $H^0(\nP^r, \sV)=H^0(\nP^r,\sW)=H^1(\nP^r,\sW)=0$. So we apply Lemma \ref{dialem} to get $\phi_0$ lifting $\id_{\sF}$. Next, we construct $\phi_1$. The restricted morphism $\phi_0|_{\sV}:\sV\longrightarrow \sW$ induces a morphism $\phi'': \sV(1)\longrightarrow \sW(1)$. Hence, we are in the situation
$$\xymatrix{
	0 \ar[r] & N(1)  \ar[r] & \sO_{\P^r}^{(r-n)(d-r+n)+1}  \ar[r]^-{v'} & \sV(1)  \ar@{^{(}->}[d]^-{\phi''}  &  \\
	0 \ar[r] & P(1)  \ar[r] &  \sO_{\P^r}^{(r-n+1)(d-r+n+g)} \ar[r]^-{w'} & \sW(1)   &
}$$
We can check that $H^0(\nP^r, N(1))=0$ and $H^0(\nP^r, P(1))=0$ by using (\ref{eq:01}) and (\ref{eq:06}). Then Lemma \ref{dialem} can be applied to lift $\phi''$ to obtain the desired $\phi_1$. Therefore, we obtain the commutative diagram (\ref{eq:07}) as claimed.

Immediately from the diagram (\ref{eq:07}), we have
$$\Coker \phi_0=\sO_{\P^r}^g\quad \text{and } \Coker \phi_1=\sO_{\P^r}^{d-r+n-1+g(r-n+1)}(-1),$$
and as a consequence of Snake lemma, we obtain a short exact sequence
\begin{equation}\label{eq:08}
0 \longrightarrow F \longrightarrow \sO_{\P^r}^{f+g}(-1) \longrightarrow \sO_{\P^r}^g \longrightarrow 0,
\end{equation}
\iffalse  %***********************************removed below
 , we obtain the following commutative diagram
\[
\xymatrix{
0 \ar[r] & \sV \ar@{^{(}->}[d]\ar[r] & \sO_{\P^r}^{d-e} \ar[r]^-v \ar@{^{(}->}[d]^-{\phi_0} & \sF \ar[r] \ar@{=}[d]^{\id_{\sF}} & 0 \\
0 \ar[r] & \sW \ar[r] & \sO_{\P^r}^{d-e+g} \ar[r]^-w & \sF \ar[r] & 0
}
\]
where the vertical maps are injective as indicated. By Snake lemma, the above diagram induces the short exact sequence
$$
0 \longrightarrow \sV \longrightarrow \sW  \longrightarrow \sO_{\P^r}^g \longrightarrow 0.
$$
Consider the following two short exact sequences in the diagram (\ref{eq:01})
$$
0 \longrightarrow N \longrightarrow \sO_{\P^r}(-1)^{e(d-e)+1} \stackrel{v'}{\longrightarrow} \sV \longrightarrow 0
$$
and

$$
0 \to P \to \sO_{\P^r}(-1)^{(e+1)(d-e+g)} \xrightarrow{w'} \sW \to 0.
$$
By twisting $\sO_{\P^r}(1)$ and applying Lemma \ref{dialem}, we get the following commutative diagram.
\begin{equation}\label{eq:02}
\begin{split}
\xymatrix{
& 0 \ar[d] & 0 \ar[d] & 0 \ar[d] & \\
0 \ar[r] & N \ar[d] \ar[r] & \sO_{\P^r}(-1)^{(r-n)(d-r+n)+1} \ar[d] \ar[r] & \text{im}(v') \ar[d] \ar[r] & 0 \\
0 \ar[r] & P \ar[d] \ar[r] &  \sO_{\P^r}(-1)^{(r+1-n)(d+n-r+g)} \ar[d] \ar[r] & \text{im}(w') \ar[d] \ar[r] & 0 \\
0 \ar[r] & F \ar[d] \ar[r] & \sO_{\P^r}(-1)^{d+n-r-1+g(r+1-n)} \ar[d] \ar[r]  &\sO_{\P^r}^g \ar[d] \ar[r] & 0 \\
& 0 & 0 & 0 &
}
\end{split}
\end{equation}
\fi %********************************************** removed above
where  $f=d-r+n-1+g(r-n)$ and $F=\Coker\phi_2$ which is a vector bundle of rank $f$ on $\P^r$.
Taking the dual of (\ref{eq:08}), we obtain a short exact sequence
\begin{equation}\label{eq:33}
0 \longrightarrow \sO_{\P^r}^g \longrightarrow \sO_{\P^r}^{f+g}(1) \longrightarrow F^* \longrightarrow 0.
\end{equation}
By Proposition \ref{regten}, $F^*$ is $(-1)$-regular. In addition, for each $i>0$, the short exact sequence (\ref{eq:33}) induces an exact complex
\begin{equation}\label{eq:12}
\begin{split}
\cdots \longrightarrow\wedge^2(\sO_{\P^r}^g)\otimes \Sym^{i-2}(\sO_{\P^r}^{f+g}(1))\longrightarrow\wedge^1(\sO_{\P^r}^g)\otimes \Sym^{i-1}(\sO_{\P^r}^{f+g}(1))\longrightarrow \hspace{2cm} \\
\longrightarrow \Sym^i(\sO_{\P^r}^{f+g}(1))\longrightarrow \Sym^i(F^*)\longrightarrow 0,
\end{split}
\end{equation}
which implies that $\Sym^i(F^*)$ is $(-i)$-regular.

Finally,  we form the following commutative diagram
\[
\xymatrix{
& & & 0 \ar[d] & \\
& 0 \ar[d] & & N \ar[d]^-{\phi_2} & \\
0 \ar[r] & \text{ker}(w'') \ar[d] \ar[r] & \sO_{\P^r}^{n_2}(-2) \ar@{=}[d] \ar[r]^-{w''} & P \ar[d] \ar[r] & 0\\
0 \ar[r] & \text{ker}(q) \ar[d] \ar[r] & \sO_{\P^r}^{n_2}(-2) \ar[r]^-{q} & F \ar[d] \ar[r] & 0 \\
& N \ar[d] & & 0 &\\
& 0
}
\]
where the morphism $q$ is induced by $w''$ composing with the projection from $P$ to $F$.
We compute a regularity bound for $N$ to finish the proof.
By a truncation of the resolution (\ref{eq:06}), it is easy to check that  $H^2(\nP^r,\text{ker}(w'')(m))=0$ for $m \in \Z$.
The Eagon-Northcott complex associated to the morphism $q$ (Lemma \ref{lm:10}) gives rise to a free resolution of $\ker(q)$
\begin{equation}\label{eq:09}
\cdots \longrightarrow \wedge^{f+2}(\sO_{\P^r}^{n_2}(-2)) \otimes \Sym^1(F^*) \otimes \det F^* \longrightarrow \wedge^{f+1}(\sO_{\P^r}^{n_2}(-2)) \otimes \det F^* \longrightarrow \text{ker}(q) \longrightarrow 0.
\end{equation}
Note that $\det F^* = \sO_{\P^r}(f+g)$. In addition, for $i\geq 0$,  the $i$-th term in the resolution (\ref{eq:09}) is of the form
 $$\wedge^{f+i+1}(\sO_{\P^r}^{n_2}(-2) ) \otimes \Sym^i(F^*) \otimes \det F^*$$
which is a direct sum of copies of $\Sym^i(F^*)(-f-2+g-2i)$.
Since we have seen that $\Sym^i(F^*)$ is $(-i)$-regular, we obtain that  $\wedge^{f+i+1}(\sO_{\P^r}^{n_2}(-2) ) \otimes \Sym^i(F^*) \otimes \det F^*$ is $(f+2-g+i)$-regular.
By Proposition \ref{regten}, we see that $\text{ker}(q)$ is $(f+2-g)$-regular. In particular, this implies that $H^1(\P^r, \text{ker}(q)(m))=0$ for all $m \geq f+1-g$. Thus $H^1(\P^r, N(m))=0$ for all $m \geq f+1-g$. Hence by (\ref{eq:10}), we deduce that
 $$X \subseteq \P^r \text{ is $m$-normal, for }m\geq f+1-g.$$
Recall that $f=d-r+n-1+g(r-n)=d-e+1+ge$.
Observe that a general hyperplane section of a scroll is again a scroll. Furthermore, we know the regularity bound for the curve case by \cite{GLP}. Hence inductively by Lemma \ref{reglem}, we obtain $$\reg(X) \leq f+2-g=d-e+1+g(e-1),$$
which finishes the proof of the theorem.
\end{proof}

%\begin{lemma}\label{dialem}
%Consider the following short exact sequences of coherent sheaves on $\P^r$
%$$
%0 \to \mathcal{A} \to \sO_{\P^r}^a \to \sF \to 0 \text{ and }
%0 \to \mathcal{B} \to \sO_{\P^r}^b \to \sG \to 0.
%$$
%Assume that $a \leq b$ and $H^0(\P^r, \mathcal{A}) = H^0(\P^r, \mathcal{B}) = H^1(\P^r, \mathcal{B}) = 0$. %If furthermore there is an injective map $\sF \hookrightarrow \sG$, then we obtain the %following commutative diagram
%\[
%\xymatrix{
%0  \ar[r] & \mathcal{A} \ar[r]\ar[d] & \sO_{\P^r}^a \ar[r]\ar[d] & \sF \ar[r]\ar[d] & 0 \\
%0 \ar[r] & \mathcal{B} \ar[r] & \sO_{\P^r}^b \ar[r] & \sG \ar[r] & 0
%}
%\]
%where horizontal sequences are exact and vertical maps are injective.
%\end{lemma}

%\begin{proof}
%We can write $\sO_{\P^r}^a = A \otimes \sO_{\P^r}$ and $\sO_{\P^r}^b = B \otimes \sO_{\P^r}$ for some vector spaces $A$ and $B$. By the cohomology vanishing conditions, we have the following maps
%$$
%A = H^0(\P^r, \sO_{\P^r}^a) \hookrightarrow H^0(\P^r, \sF) \hookrightarrow H^0(\P^r, \sG) = H^0(\P^r, \sO_{\P^r}^b) = B.
%$$
%Thus we get an injective map $A \hookrightarrow B$ which induces an injective map $\sO_{\P^r}^a \hookrightarrow \sO_{\P^r}^b$. It is easy to check that the following diagram is commutative.
%\[
%\xymatrix{
% \sO_{\P^r}^a \ar[r]\ar[d] & \sF \ar[d] \\
% \sO_{\P^r}^b \ar[r] & \sG
%}
%\]
%Now the assertion follows from the Snake lemma.
%\end{proof}

\begin{remark}\label{bertin}
	The idea in the proof of Theorem \ref{main1} follows the one in Bertin's work \cite{B}, which goes back to the work of \cite{GLP}. We choose the same line bundle $A$ as in \cite{B} to deduce a locally free resolution of $t^*A$. However, in \cite[p.180]{B}, the sheafification of a minimal free resolution of the cokernel sheaf $\sF$ was claim to have the form
	\begin{equation}\label{eq:13}
		\cdots\longrightarrow \sO_{\nP^r}^{(r-n+1)(d+n-r)}(-1) \longrightarrow \sO_{\P^r}^{d-r+n} \longrightarrow \sF \longrightarrow 0,
	\end{equation}
	from which the conjectured regularity bound was proved. Unfortunately, the resolution  (\ref{eq:13}) only works when genus $g=0$. For arbitrary genus $g$, since $h^0(\nP^r, \sF)=d-r+n+g$, the complex (\ref{eq:13}) cannot be deduced from a minimal free resolution. The correct free resolution is (\ref{eq:06}) but the difference between $d-r+n$ and $d-r+n+g$ makes the proof complicated.
\end{remark}

\begin{remark}\label{rmk:01}
We can use the same argument to generalize Theorem \ref{main1} to singular case under certain conditions.  For the convenience of the reader, we formulate this direction here and leave the details. Assume that $C$ is a smooth projective curve of genus $g\geq 0$ and let $E$ be a vector bundle on $C$ of rank $n$ and degree $d$ (not necessarily ample). Assume that $V \subseteq H^0(\P(E), \sO_{\P(E)}(1))$ is a base-point-free subspace of dimension $r+1$ and it induces a birational morphism $p \colon \nP(E) \rightarrow \nP^r=\nP(V)$ to the image. This time we denote by $X$ the image of $p$. So we have the diagram as follows:
	\[
	\begin{split}
	\xymatrix{
		\P(E) \ar[d]_t \ar[r]^-p& X \subset \P^r=\nP(V)\\
		C &
	}
	\end{split}.
	\]
We regard $X$ as a generalized scroll. In order to get a regularity bound for $X$, we still use a general line bundle $A=\sO_C(D)$ for an effective divisor $D$ consisting of general $d-r+n+g$ points. The global section $s_D\in H^0(C, A)$ gives rise to a short exact sequence
$$
0 \longrightarrow \sO_C \stackrel{\cdot s_D}{\longrightarrow} A \longrightarrow A_D \longrightarrow 0.
$$
Under the following two conditions:
\begin{itemize}
	\item [(1)] $p_*\sO_{\nP(E)}=\sO_X$,
	\item [(2)] the induced morphism $R^1p_*(t^*\sO_C)\rightarrow R^1p_*(t^*A)$ is injective,
\end{itemize}
 the proof of Theorem \ref{main1} works without much change and we can obtain the same regularity bound. Note that the condition (2) is equivalent to
$$
H^1(\P(E), \sO_{\P(E)}(m))=H^1(C, \Sym^m(E))=0 \text{ for } m \gg 0,
$$
which may not hold in general if $E$ is not ample and $g>0$.
%Since it is not clear to us what the actual geometric meaning of these two technical conditions are, we prefer to state out results in the smooth setting which is more transparent.
\end{remark}
%\begin{remark}\label{char=0}
%	The assumption $\text{char}(k)=0$ is only used when we compute the regularity of $S^r(F^*)$ by using Lemma \ref{regten} (1). If $g=0$, then $F = \sO_{\P^r}(-1)^{d+n-r-1+g(r+1-n)}$ so that we do not need the assumption $\text{char}(k)=0$.
%end{remark}

\begin{remark}\label{rmk:02}
	\cite{GLP} and \cite{N1} show a different choice of the line bundle $A$ which gives a better regularity bound for curves. It is achieved by the key observation that a minimal free resolution of such line bundle $A$ has a good shape (see \cite[Lemma 2.3]{GLP} and \cite[Lemma 4]{N1}).
However, for the higher dimensional scroll cases, it seems that such choice of $A$ does not work.
For instance, consider an elliptic surface scroll $S \subset \P^4$ of degree 5. Then the choice of $A$ is nothing but the structure sheaf $A=\sO_S$.
It is expected that the sheafification of a minimal free resolution of $\sO_S$ is of the form
	$$
	\cdots \longrightarrow \sO_{\P^4}^{n_1}(-1) \longrightarrow \sO_{\P^4} \longrightarrow \sO_S \longrightarrow 0.
	$$
However, the actual sheafification of a minimal free resolution of $\sO_S$ is
	$$
	\cdots \longrightarrow \sO_{\P^4}^5(-2) \longrightarrow \sO_{\P^4} \longrightarrow \sO_S \longrightarrow 0.
	$$
\end{remark}

$ $
%%%%%%%%%%%%%%%%%%%%%%%%%%%%%%%%%%%%%%%%%%%%%%%%%%%%%%%%%%%%%%%%%%%%%%%%%%%%%%%%%%%%%%%%%%%%%%%%%%%%%%%%
%BIBLIOGRAFIA

\end{document}